\documentclass[twoside]{amsart}
\usepackage{latexsym}
\usepackage{amssymb,amsmath,amsopn}
\usepackage[dvips]{graphicx}   

\newtheorem{thm}{Theorem}
\newtheorem{lem}{Lemma}
\theoremstyle{definition}
\newtheorem{defn}{Definition} 

\newtheorem{ques}{Question}

\renewcommand{\Re}{\mathbb R}

\newcommand{\M}{\mathcal{M}}

\newcommand{\F}{\mathfrak{F}}

\DeclareMathOperator{\inter}{int}
\DeclareMathOperator{\bd}{bd}

\DeclareMathOperator{\conv}{conv}

\DeclareMathOperator{\dif}{d}

\begin{document}
\title{Ellipsoid characterization theorems}

\author[Z. L\'angi]{Zsolt L\'angi}
                                                                               
\address{Zsolt L\'angi, Dept. of Geometry, Budapest University of Technology,
Budapest, Egry J\'ozsef u. 1., Hungary, 1111}
\email{zlangi@math.bme.hu}
                                                                        
\thanks{{2010 \sl Mathematics Subject Classification.} 52A21, 52A20, 46C15}
\keywords{normed space, ellipsoid, tangent segment, billiard reflection, Busemann angular bisector,
Glogovskij angular bisector, billiard angular bisector.}

\begin{abstract}
In this note we prove two ellipsoid characterization theorems.
The first one is that if $K$ is a convex body
in a normed space with unit ball $M$, and for any point $p \notin K$
and in any $2$-dimensional plane $P$ intersecting $\inter K$ and containing $p$,
there are two tangent segments of the same normed length from $p$
to $K$, then $K$ and $M$ are homothetic ellipsoids.
Furthermore, we show that if $M$ is the unit ball of a strictly convex, smooth norm, and
in this norm billiard angular bisectors coincide with Busemann angular bisectors or Glogovskij
angular bisectors, then $M$ is an ellipse.
\end{abstract}
\maketitle

\section{Introduction}

The problem what geometric or other properties distinguish Euclidean spaces from the other normed spaces
is extensively studied in convex geometry and functional analysis (cf. for instance, the book \cite{A86}).
In this note we present two more such properties.

In our investigation, $\Re^n$ denotes the standard $n$-dimensional real linear space, with origin $o$.
For points $p, q \in \Re^n$, $[p,q]$ denotes the closed segment with endpoints $p$ and $q$,
and $[ p,q \rangle$ denotes the closed ray emanating from $p$ and containing $q$.
If $M$ is an $o$-symmetric convex body, then $\M$ is the norm induced by $M$, and $||p||_\M$ is the norm of $p$.
If it is obvious which norm we mean, we may write simply $||p||$.

Let $K$ be a convex body in $\Re^n$ with $n \geq 2$.
A closed nondegenerate segment $[p,q]$, contained in a supporting line of $K$
(also called a \emph{tangent line} of $K$) and with $q \in K$, is a \emph{tangent segment of $K$ from $p$}.
If for some norm $\M$ and convex body $K \subset \Re^n$, for any point $p \notin K$
and in any $2$-dimensional plane $P$ intersecting $\inter K$ and containing $p$,
there are two tangent segments from $p$ to $K$ of the same length in $\M$, 
we say that $K$ satisfies the \emph{equal tangent property} with respect to $\M$.
It is well-known that Euclidean balls have this property with respect to the standard Euclidean norm,
and it has been proved that, among convex bodies, only they have this property
(cf. \cite{GO99}).
On the other hand, Wu \cite{W08} showed that for any norm $\M$,
if the unit ball $M$ has the equal tangent property with respect to $\M$, then $M$ is an ellipsoid.

Our first theorem is a combination of these results.
We note that our definitions of tangent segment and equal tangent property are slightly more
general than the one in \cite{W08}.
On the other hand, for a strictly convex body $K$, equal tangent property with respect to $\M$
simplifies to the property that for any $p \notin K$, all tangents from $p$ to $K$ have
equal length with respect to $\M$.

\begin{thm}\label{thm:equaltangent}
Let $K \subset \Re^n$ be a convex body and let $\M$ be a norm.
If $K$ satisfies the equal tangent property with respect to $\M$, then $K$ and $M$ are homothetic ellipsoids.
\end{thm}

Our second theorem involves angular bisectors in normed planes.
Based on the properties of angular bisectors in the Euclidean plane, these bisectors can be defined
in several ways. Here we mention two of them (cf. \cite{B75} and \cite{G70}, and for a detailed discussion
on bisectors, \cite{MS04}).

\begin{defn}
Let $[a,a+b\rangle, [a,a+c\rangle \subset \Re^2$ be closed rays, and let $\M$ be a norm.
The \emph{Busemann angular bisector} of $[a,a+b\rangle$ and $[a,a+c\rangle$ with respect to $\M$ is
\[
A_B([a,a+b\rangle , [a,a+c\rangle) = \left[ a,a + \frac{b}{||b||} + \frac{c}{||c||} \right\rangle .
\]
\end{defn}

For the following definition we note that the distance of a point $p \in \M$ and a set $S \subset \M$ is
$\inf \{ ||p-q||_\M : q \in S \}$.

\begin{defn}
Let $[a,a+b\rangle, [a,a+c\rangle \subset \Re^2$ be closed rays, and let $\M$ be a norm.
The \emph{Glogovskij angular bisector} $A_G([a,a+b\rangle , [a,a+c\rangle)$
of $[a,a+b\rangle$ and $[a,a+c\rangle$ with respect to $\M$ is
the set of points in $\conv ([a,a+b\rangle, [a,a+c\rangle )$ that are equidistant from 
$[a,a+b\rangle$ and $[a,a+c\rangle$ in $\M$.
\end{defn}

We remark that a point $p$ of the convex hull of the two rays is equidistant from the rays if, and only if,
$p+\lambda M$ touches the rays for some $\lambda \geq 0$, and thus, for any two rays, their
Glogovskij bisector is a closed ray.

In this note we define another angular bisector with respect to a strictly convex, smooth norm.
To do this, we recall the notion of billiard reflection in a normed plane.

A billiard path in a Riemannian, or in particular in a Euclidean space,
is a piecewise smooth geodesic curve with the nonsmooth points
contained in the boundary of the space, where the smooth arcs are connected via the reflection law:
the angle of reflection is equal to the angle of incidence.
A generalization of billiard paths for Finsler geometries was introduced by Gutkin and Tabachnikov \cite{GT02},
who replaced the reflection law by the \emph{least action principle}, namely,
that two points are connected by the shortest piecewise smooth geodesic arc containing a point of the boundary.

Formally, let $\M$ be a strictly convex, smooth norm in $\Re^2$, $L$ be a line, and let
$[a,b\rangle$ and $[a,c\rangle$ be closed rays, with $a \in L$, that are not separated by $L$.
If $a$ is a critical point of the distance function $F(z) = ||z-b|| + ||z-c||$, where $z \in L$,
then we say that $[a,c\rangle$ is a \emph{billiard reflection} of $[a,b\rangle$ on $L$ with respect to $\M$.

Note that since $\M$ is strictly convex and smooth, the distance function has a unique critical point,
which is independent of the choice of $b$ and $c$.
Furthermore, for any two closed rays emanating from the same point,
there is a unique line on which they are a billiard reflection, which coincides with the tangent line
of any metric ellipse with its foci on the two rays. For reference on metric ellipses and conic sections
in normed planes, the interested reader is referred to \cite{HM11}.

Billiard reflections can be interpreted geometrically as well (cf. Corollary 3.2 and Proposition 5.1 of \cite{GT02}):
if $L$ is a line containing $o$, then $[o,c\rangle$ is a billiard reflection of $[o,b\rangle$ if, and only if,
the tangent lines of $M$ at $-\frac{b}{||b||}$ and $\frac{c}{||c||}$ meet on $L$ (cf. Figure~\ref{fig:reflection}).

We remark that billiard reflections in \cite{GT02} were defined only for norms with unit disks satisfying
the $C^\infty$ differentiability property, and thus, the authors of \cite{GT02} gave a geometric interpretation only for such norms.
On the other hand, recall that the subfamily of strictly convex $C^\infty$ norms is everywhere dense
in the family of strictly convex, smooth norms (cf. \cite{G07}).
Hence, by continuity, the equivalence of our definition with the geometric interpretation in Figure~\ref{fig:reflection} holds for
any strictly convex, smooth norm.

\begin{figure}[here]
\includegraphics[width=0.7\textwidth]{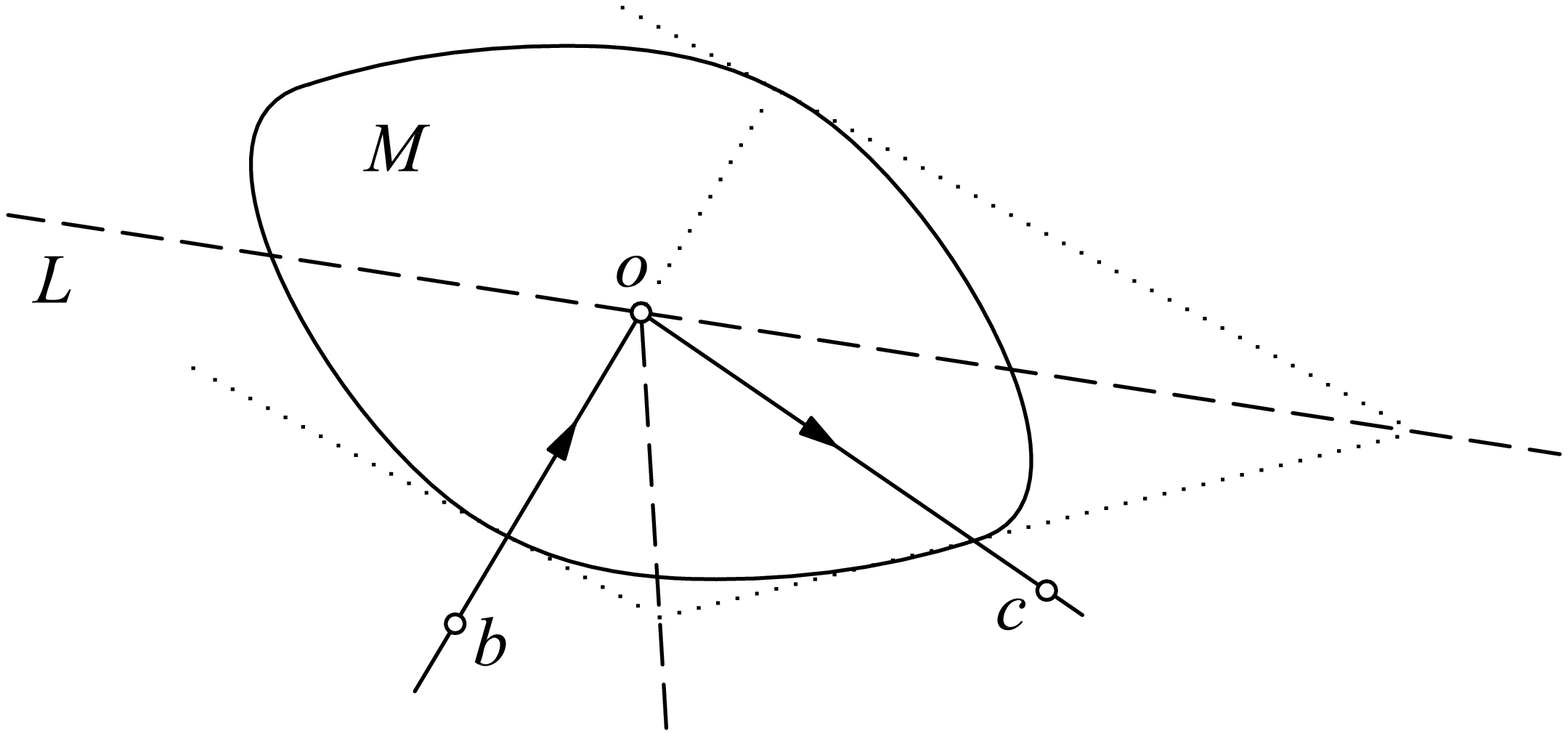}
\caption[]{}
\label{fig:reflection}
\end{figure}

\begin{defn}
Let $[a,b\rangle, [a,c\rangle \subset \Re^2$ and let $\M$ be a strictly convex, smooth norm.
We say that $L$ is an \emph{external billiard bisector} of $[a,b\rangle$ and $[a,c\rangle$
with respect to $\M$, if $[a,c\rangle$ is a billiard reflection of $[a,b\rangle$
on $L$ with respect to $\M$.
The closed ray of the external billiard bisector of $[a,b\rangle$ and $[a,2a-c\rangle$,
contained in $\conv([a,b\rangle \cup [a,c\rangle)$ is called the \emph{internal billiard bisector}
or shortly \emph{billiard bisector} of $[a,b\rangle$ and $[a,c\rangle$, denoted by
$A_b([a,b\rangle,[a,c\rangle )$.
\end{defn}

In Figure~\ref{fig:reflection}, the two dashed lines show the external and the internal billiard bisectors
of the rays $[o,b\rangle$ and $[o,c\rangle$.

Our second result is the following.

\begin{thm}\label{thm:billiard}
Let $\M$ be a strictly convex, smooth norm.
Then the following are equivalent. 
\begin{itemize}
\item[\ref{thm:billiard}.1] For every $[a,b\rangle, [a,c\rangle \subset \Re^2$, we have
$A_b([a,b\rangle, [a,c\rangle) = A_G([a,b\rangle, [a,c\rangle)$.
\item[\ref{thm:billiard}.2] For every $[a,b\rangle, [a,c\rangle \subset \Re^2$, we have
$A_b([a,b\rangle, [a,c\rangle) = A_B([a,b\rangle, [a,c\rangle)$.
\item[\ref{thm:billiard}.3] $M$ is an ellipse.
\end{itemize}
\end{thm}

We note that, by a result of D\"uvelmeyer \cite{D04}, Busemann and 
Glogovskij angular bisectors coincide if, and only if, the plane is a Radon plane.

The definition of billiards for normed planes gives rise to numerous possible questions, among which we mention one.
It is well-known that in the Euclidean plane the \emph{pedal triangle} of an acute triangle $T$
(that is the triangle connecting the base points of the altitudes of $T$)
is a billiard triangle; in fact, this is the only billiard triangle in $T$.
This billiard triangle is called the \emph{Fagnano orbit}.
Using the definition of billiard reflections, Fagnano orbits can be generalized
for normed planes (cf. \cite{GT02}).
On the other hand, we may define the altitudes of $T$ in a natural way:
For a triangle $T$ with vertices $a,b$ and $c$, $[a,h_a]$ is called an \emph{altitude} of $T$,
if, in terms of Birkhoff orthogonality, $a-h_a$ is \emph{normal} to $b-c$
(for Birkhoff orthogonality, cf. \cite{T96}).

Our question is the following.

\begin{ques}
Let $\M$ be a strictly convex, smooth norm. Prove or disprove that if the pedal triangle
of any triangle $T$, if it exists, is a billiard triangle in $T$, then $M$ is an ellipse.
\end{ques}

In Sections~\ref{sec:equaltangent} and \ref{sec:billiard} we present the proofs of
Theorems~\ref{thm:equaltangent} and \ref{thm:billiard}, respectively.

\section{Proof of Theorem~\ref{thm:equaltangent}}\label{sec:equaltangent} 

In the proof we use the following definition and lemma.
First, set $W = \{ (x,y) \in \Re^2: x < 1, y < 1, x+y > 1 \}$.

\begin{defn}
Let $\F$ denote the family of curves, with endpoints $(0,1)$ and $(1,0)$,
that are contained in the triangle $W$ and can be obtained as
an arc of the boundary of a strictly convex, smooth disk
with supporting half planes $x \leq 1$ and $y \leq 1$.
\end{defn}

\begin{lem}\label{lem:ellipse}
If $G \in \F$ is an integral curve of the differential equation
\begin{equation}\label{eq:C2gen}
y' (1-x)(x-y+1) + (1-y)(y-x+1) = 0
\end{equation}
for $0 < x < 1$, then $G$ is a conic section arc satisfying
\begin{equation}\label{eq:solution}
x^2+y^2-1+c(x-1)(y-1)=0
\end{equation}
for some $c \in \Re$ with $c < 2$.
\end{lem}

\begin{proof}
Note that as $G$ is an arc of the boundary of a smooth convex body, the function representing it
is continuously differentiable for $0 < x < 1$.
We distinguish three cases.

\emph{Case 1}, $x^2+y^2 > 1$ in some subinterval of $(0,1)$.
Multiplying both sides of (\ref{eq:C2gen}) by $(x^2+y^2-1)^{-3/2}$, we may write the equation in the differential form
\begin{equation}\label{eq:positive}
(x^2+y^2-1)^{-3/2} (1-x)(x-y+1) \dif y + (x^2+y^2-1)^{-3/2} (1-y)(y-x+1) \dif x = 0.
\end{equation}

Note that the coefficients of (\ref{eq:positive}) are continuously differentiable functions of $(x,y)$
in the (simply connected) subset of $W$ satisfying $x^2+y^2 > 1$.
Furthermore, it satisfies Schwarz's Theorem, and thus, by Poincar\'e's Lemma, it has a potential function $F(x,y)$.

Using standard calculus, it is easy to check that any continuously differentiable potential function
satisfies
\begin{equation}\label{eq:implicite}
(1-\bar{c}^2)x^2+(1-\bar{c}^2)y^2+2xy-2x-2y+1+\bar{c}^2 = 0
\end{equation}
for some $\bar{c} \in \Re$.
Thus, $G$ is a conic section arc.
An elementary computation shows that the eigenvalues of its matrix are $2-\bar{c}^2$ and $-\bar{c}^2$ with eigenvectors
$(1,1)$ and $(-1,1)$, respectively.
Thus, it defines an ellipse if $\bar{c}^2>2$, a parabola if $\bar{c}^2=2$, a hyperbola
with a convex arc in $0 < x < 1$ if $1 < \bar{c}^2 < 2$, the union of the lines $y=1$ and $x=1$ if $\bar{c}^2=1$,
and a hyperbola with a concave arc in $0 < x < 1$ if $0 \leq \bar{c}^2 < 1$.
These last two possibilities are clearly impossible, thus we have $1<\bar{c}^2$.
Now, introducing the notation $c = \frac{2}{1-\bar{c}^2} < 0$, we may rewrite (\ref{eq:implicite})
as in (\ref{eq:solution}) with $c < 0$.

Note that these integral curves have no common points with the circle $x^2 + y^2 = 1$, and thus,
by the continuity of $G$, we have that any curve $G$ satisfying the condition of Case 1 is an arc of
one of the curves in (\ref{eq:solution}) with some $c < 0$.

\emph{Case 2}, $x^2+y^2 < 1$ in some subinterval of $(0,1)$.
After multiplying both sides by $(1-x^2-y^2)^{-3/2}$, we may apply an argument similar to the one in Case 1
to obtain the curves in (\ref{eq:solution}) with $0< c < 2$.

\emph{Case 3}, otherwise.
Consider a subinterval of $(0,1)$. By our condition, in this interval $G$
contains a point of the unit circle $x^2+y^2 = 1$.
Thus, the $x$-coordinates of the intersection points of $G$ and the circle is an everywhere dense set in $(0,1)$,
which, by the continuity of $G$, yields that $G$ is an arc of $x^2+y^2=1$.
\end{proof}

\begin{proof}[Proof of Theorem~\ref{thm:equaltangent}]
First, we show the assertion for $n=2$.
To show that $K$ is strictly convex, we may use the idea in \cite{W08}.
More specifically, suppose for contradiction that $K$ is not strictly convex,
and consider a segment $[p,q] \subset \bd K$
that is maximal with respect to inclusion. Without loss of generality, we may assume that the line containing
$[p,q]$ does not separate $K$ and $o$.
Then the point $r = \lambda \left(\frac{1}{3}p + \frac{2}{3}q \right)$ is an exterior point of $K$
for any $\lambda > 1$. On the other hand, if $\lambda$ is sufficiently close to $1$,
then there are no two equal tangent segments of $K$ from $r$; a contradiction.

Now we show that $K$ is smooth.
Suppose for contradiction that $p$ is a nonsmooth point of $\bd K$; that is, there are two lines $L_1$ and $L_2$ tangent to $K$ at $p$.
Consider two parallel supporting lines $L$ and $L'$ of $K$ which intersect both $L_1$ and $L_2$.
Let $q$ and $q'$ be the tangent points of $L$ and $L'$, respectively, and for $i \in \{ 1, 2\}$, let
$r_i$ and $r_i'$ be the intersection points of $L_i$ with $L$ and $L'$, respectively.
We choose the indices in a way that $r_1 \in [q,r_2]$, which yields that $r_2' \in [q,r_1']$ (cf. Figure~\ref{fig:nonsmooth}).
Observe that, applying the equal tangent property to $r_1$ and $r_2$, we have that 
\begin{equation}\label{eq:lefthandside}
||r_1-r_2||_\M + ||r_1-p||_\M = ||r_2-p||_\M .
\end{equation}
We obtain similarly that
\begin{equation}\label{eq:righthandside}
||r_1'-r_2'||_\M + ||r_2'-p||_\M = ||r_1'-p||_\|M .
\end{equation}

\begin{figure}[here]
\includegraphics[width=0.6\textwidth]{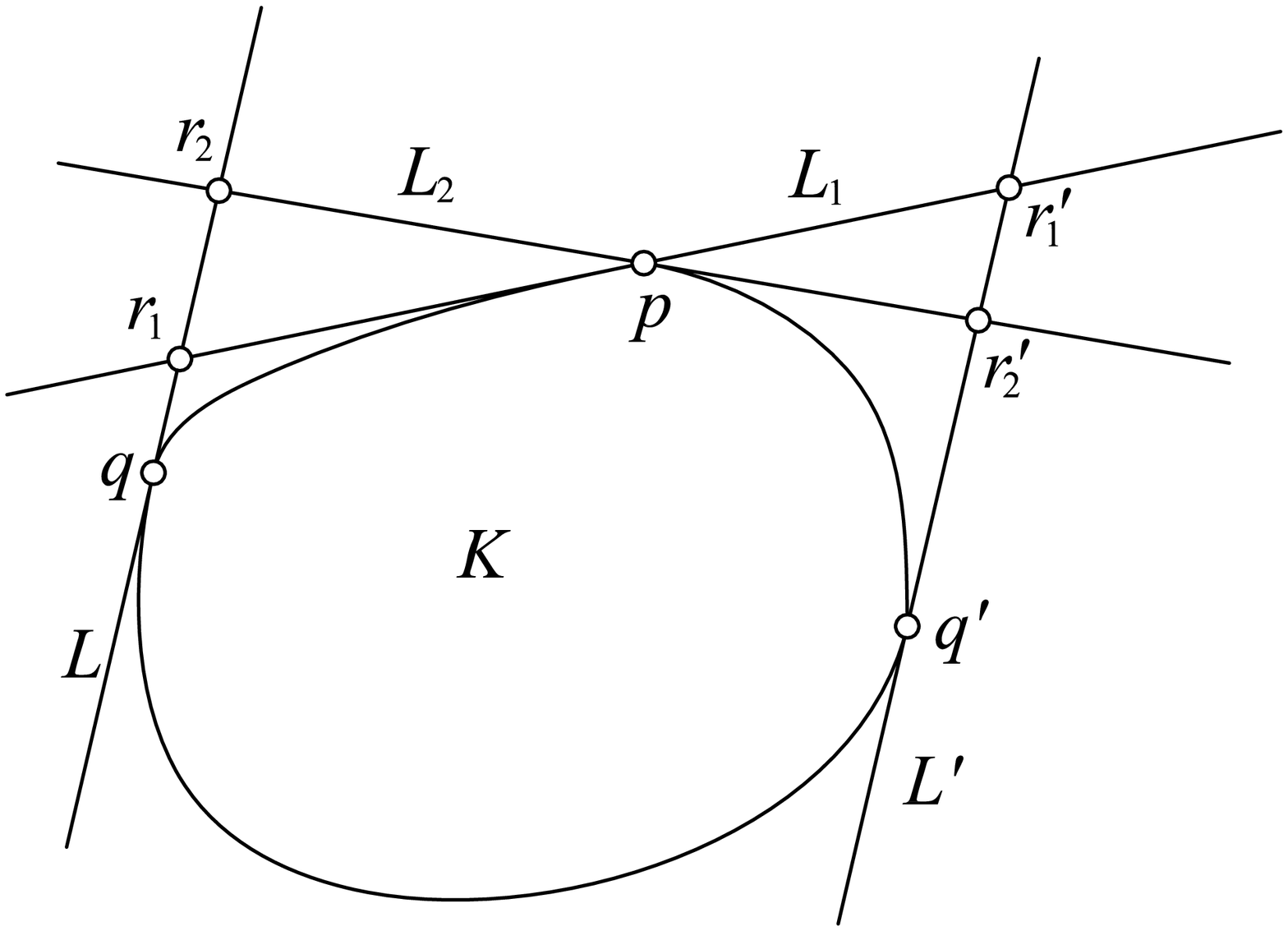}
\caption[]{}
\label{fig:nonsmooth}
\end{figure}

Clearly, the triangle  $\conv \{ p,r_1',r_2'\}$ is a homothetic copy of $\conv \{ p,r_1,r_2\}$, and thus, a there is a value $\lambda > 0$
such that $||r_1'-p||_\M = \lambda ||r_1-p||_\M$, $||r_2'-p||_\M = \lambda ||r_2'-p||_\M$ and $||r_2'-r_1'||_\M = \lambda ||r_2-r_1||_\M$.
After substituting these expressions into (\ref{eq:righthandside}) and simplifying by $\lambda$,
we immediately obtain that $||r_1-r_2||_\M + ||r_2-p||_\M = ||r_1-p||_\|M$, which contradicts (\ref{eq:lefthandside}).

In the remaining part, we assume that $K$ is strictly convex and smooth.
Consider a parallelogram $S$ circumscribed about $M$ with the property that the midpoints of its sides
belong to $M$; such a parallelogram exists: for instance, a smallest area parallelogram circumscribed about $M$
has this property.

Note that for any affine transformation $h$, $K$ has the equal tangent property with respect to $M$ if and only
if $h(K)$ has this property with respect to $h(M)$.
Hence, we may assume that $S$ is the square with vertices $(1,1)$,$(-1,1)$, $(-1,-1)$ and $(1,-1)$.

Consider the axis-parallel rectangle $R$ circumscribed about $K$.
Note that by the equal tangent property, the two tangent segments from any vertex of $R$ are of equal Euclidean length.
Thus, $R$ is a square, and the tangent points on its sides are symmetric to its two diagonals.
Since any homothetic copy of $K$ satisfies the equal tangent property, we may assume that
the points $(1,0)$ and $(0,1)$ are common points of $\bd K$ and $\bd R$ and that
the half planes $x \leq 1$ and $y \leq 1$ support $K$.

Consider a point $p_0=(x_0,y_0)$ of $\bd K$ with $0 < x_0 < 1$ and $y_0 > 0$,
and let $s_0<0$ denote the slope of the tangent line $L_0$ of $K$ at $p_0$.
The equation of $L_0$ is:
\[
y = s_0 (x-x_0) + y_0 .
\]
This line intersects the line $x=1$ at the point $q_1 = (1,y_0+s_0(1-x_0))$ and the line $y=\beta$
at $q_2 = \left( x_0 + \frac{\beta-y_0}{s_0},\beta \right)$.

\begin{figure}[here]
\includegraphics[width=.8\textwidth]{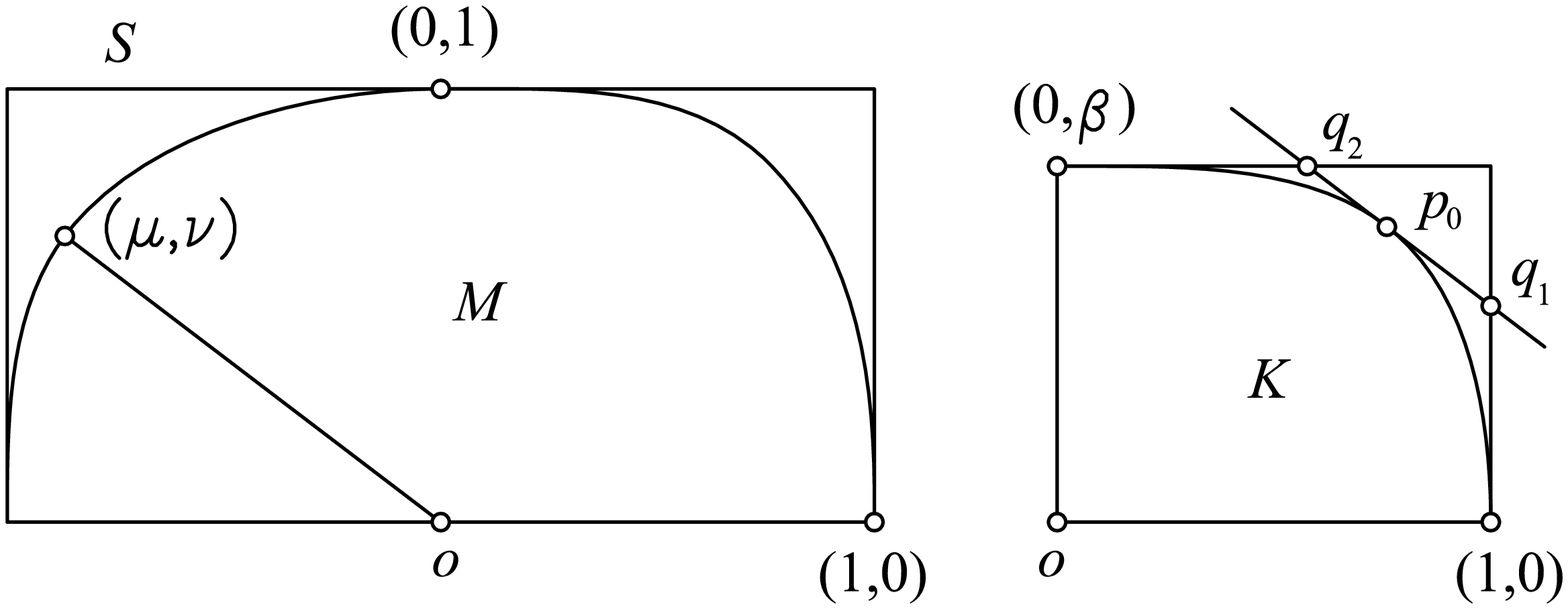}
\caption[]{}
\label{fig:tangentode}
\end{figure}

Let $(\mu,\nu)\in \bd M$ be the point with $\nu = s_0 \mu$ and $\nu > 0$ (cf. Figure~\ref{fig:tangentode}).
Thus, for the normed lengths of the two tangent segments of $K$ from $q_1$, we have
\[
y_0+s_0(1-x_0) = \frac{1-x_0}{-\mu}
\]
and for the ones from $q_2$, we obtain similarly that
\[
 x_0 + \frac{1-y_0}{s_0} = \frac{1-y_0}{s_0 \mu}
\]

Eliminating $\mu$, we have
\[
s_0 \left( \frac{x_0}{1-y_0} + 1\right) = - \left( \frac{y_0}{1-x_0} + 1 \right),
\]
which yields that the function representing the arc of $\bd K$ with $0<x<1$ and $y>0$
satisfies the differential equation
\[
y' (1-x)(x- y + 1) + (1-y)(y - x + 1) = 0.
\]

To determine $\bd K$, we need to find the integral curves of this differential equation
in $\F$.
By Lemma~\ref{lem:ellipse}, we have that the arc of $\bd K$ between $(1,0)$ and $(0,1)$ is
of the form $x^2+y^2-1+c(x-1)(y-1)=0$ for some $c < 2$, and thus, in particular, it is symmetric about
the line $y=x$.

Consider the point $u=(\zeta,\omega)$ of this arc, and let $t=(1,\tau)$ be the point where the tangent
line of $K$ at $u$ intersects the line $x=1$.
Then, by the axial symmetry of the arc, we have that the tangent line of $K$ at $u'=(\omega,\zeta)$
intersects the line $y=1$ at $t'=(\tau,1)$.
Note that by the equal tangent property, the $M$-lengths of the segments $[u,v]$ and $[u',v']$ are both equal to $\tau$.
Since they are of equal Euclidean lengths as well, it follows that the radii of $M$ in their directions are of equal Euclidean lengths,
which yields that the union of the arcs of $\bd M$ in the second and the fourth quadrants is symmetric about the line $y=x$.
Thus, the two supporting lines of $M$, parallel to the line $y=x$, touch $M$ at points on the line $y=-x$.
We may obtain similarly that the supporting lines of $M$, parallel to the line $y=-x$, touch $M$ at points on the line $y=x$.
Then the rectangle $S'$, with these four supporting lines as its sidelines, satisfies the property that the midpoints of its sides belong to $M$.
Hence, we may repeat our argument with $S'$ in place of $S$, and obtain that $\bd K$ is the union of another four conic section arcs.
Recall that conic section arcs are analytic, and thus, we have that the four arcs belong to the same conic section, which yields, by the compactness of $K$,
that $K$ is an ellipse.

Now, consider an affine transformation $h$ that transforms $K$ into a Euclidean disk.
We may assume that the centre of both $h(K)$ and $h(M)$ is the origin.
Consider any line $L$ passing through $o$.
If $p$ is a point on $L$ with $p \notin h(K)$ such that the two tangent lines of $h(K)$
containing $p$ touch $h(K)$ at $q_1$ and $q_2$,
then the Euclidean lengths of the segments $[p,q_1]$ and $[p,q_2]$ are equal.
Since their $h(M)$-normed lengths are also equal, we have that the two radii of $h(M)$ in the directions
of $[p,q_1]$ and $[p,q_2]$ have equal Euclidean lengths.
As $p$ was an arbitrary point, it follows that $h(M)$ is axially symmetric to $L$.
Thus, $h(M)$ is axially symmetric to any line passing through $o$, and hence it is a Euclidean disk.
This yields the assertion for the plane.

Finally, we examine the case $n > 2$.
Consider any section of $K$ with a plane $P$, and the intersection of $M$ with the plane $P'$ parallel to $P$
that contains the origin.
We may apply our previous argument, and obtain that these planar sections are homothetic ellipses.
Thus, a well-known theorem of Burton \cite{B76} yields that $M$ and $K$ are ellipsoids,
from which the assertion readily follows.
\end{proof}

\section{Proof of Theorem~\ref{thm:billiard}}\label{sec:billiard}

Clearly, if $M$ is an ellipse, then all the three angular bisectors coincide.

First, let $\M$ be a norm in $\Re^2$, and assume that
$A_b([a,b\rangle, [a,c\rangle) = A_G([a,b\rangle, [a,c\rangle)$ for any
$[a,b\rangle, [a,c\rangle \subset \Re^2$.

Since $M$ is centrally symmetric, there is a parallelogram $S$ such that the midpoints of its sides belong to $M$.
We may apply an affine transformation such as in the proof of Theorem~\ref{thm:equaltangent}, and hence,
we may assume that $S$ is the square with vertices $(1,1)$, $(1,-1)$, $(-1,-1)$ and $(-1,1)$.

Let the arc of $\bd M$ in the quadrant $x\geq 0$ and $y \geq 0$ be defined as the graph
of a function $y=f(x)$.
Note that $f$ is continuously differentiable for $0<x<1$, and that
$f(1)=0$, $f(0)=1$ and $f'(0) = 0$.

Consider the point $p_0=(x_0,f(x_0))$, where $0<x_0<1$.
Let $f(x_0) = y_0$, and $f'(x_0) = s_0$.
Then the equation of the tangent line of $M$ at $p_0$ is
\[
y-y_0 = s_0(x-x_0)
\]
Let $q_1$ and $q_2$ be the points where this line intersects the lines $x=1$ and $y=1$, respectively.
Then
\[
q_1 = \big( 1,y_0+s_0(1-x_0) \big) \textrm{ and } q_2 = \left(x_0+\frac{1-y_0}{s_0}, 1 \right) .
\]

Observe that $q_1$ is equidistant from the $x$-axis and the line $y = \frac{y_0}{x_0} x$.
Clearly, the normed distance of $q_1$ and the $x$-axis is $y_0 + s_0(1-x_0) > 0$.
Thus, the disk $q_1 + (y_0 + s_0(1-x_0))M$ touches the line $y = \frac{y_0}{x_0} x$,
or in other words, the line $L_1$ with equation
$y+1=\frac{y_0}{x_0} \left( x+\frac{1}{y_0+s_0(1-x_0)} \right)$ touches $M$.

From the observation that $q_2$ (and thus also $-q_2$) is equidistant from the $y$-axis and the line  $y= \frac{y_0}{x_0} x$,
we obtain similarly (for $-q_2$) that the line $L_2$ with equation $y - \frac{1}{x_0 + \frac{1-y_0}{s_0}} = \frac{y_0}{x_0} (x-1)$ touches $M$.
It is easy to check that then $L_1$ and $L_2$ coincide, and thus, that their $y$-intercepts are equal.
Hence,
\[
\frac{y_0}{x_0} \cdot \frac{y_0+s_0(1-x_0)+1}{y_0+s_0(1-x_0)} = \frac{s_0x_0+1-y_0+s_0}{s_0 x_0+1-y_0},
\]
which yields the differential equation
\[
\frac{y}{x} \cdot \frac{y+y'(1-x)+1}{y+y'(1-x)} = \frac{y'x +1-y+y'}{xy'+1-y}.
\]

Simplifying, we obtain that
\[
\big(y' (1-x)(x-y+1) + (1-y)(y-x+1)\big)(y' x -y) = 0
\]

If the second factor is zero in some interval, then in this interval $y=cx$ for some constant $c \in \Re$.
This curve cannot be contained in the boundary of an $o$-symmetric convex disk.
Thus, the first factor is zero in an everywhere dense set,
which yields, by continuity, that it is zero for every $0<x<1$.
Hence, by Lemma~\ref{lem:ellipse}, we obtain that the part of $\bd M$ in the first quadrant
is a conic section arc. Applying a similar argument for the arcs in the other quadrants,
we obtain that $\bd M$ is the union of four conic section arcs,
with the lines $y=x$ and $y=-x$ as axes of symmetry.
We remark that these arcs are analytic curves.

Let $R$ be the rectangle, with its sides parallel to the lines $y=x$ and $y=-x$,
that is circumscribed about $M$.
Observe that the midpoints of the sides of $R$ are points of $M$.
Thus, we may apply our argument for $R$ in place of $S$, which yields, in particular, that $\bd M$ is analytic;
that is, that $\bd M$ belongs to the same conic section: an ellipse that contains the points $(1,0)$, $(0,1)$,
$(-1,0)$ and $(0,-1)$ and is contained in the square $S$.
There is only one such ellipse: the circle with equation $x^2+y^2 = 1$; from which the assertion follows.

We are left with the case that (\ref{thm:billiard}.2) yields (\ref{thm:billiard}.3).
Assume that  $A_b([a,b\rangle, [a,c\rangle) = A_B([a,b\rangle, [a,c\rangle)$ for any
$[a,b\rangle, [a,c\rangle \subset \Re^2$.
Consider any point $p \in \bd M$. Without loss of generality, we may assume that $p=(1,0)$,
and that the lines $x=1$ and $y=1$ are tangent lines of $M$.
Since $M$ is smooth, for $-1< x<1$ and  $y > 0$, $\bd M$ is the graph of
a continuously differentiable function $y=f(x)$.
Then, with the notations $y_0 = f(x_0)$ and $s_0 = f'(x_0)$,
the equation of the tangent line $L$ of $M$ at the point $q=(x_0,y_0)$
intersects the line $x=1$ at the point $u=(1,y_0+s_0(1-x_0))$.
This point is on the billiard bisector of $[o,p\rangle$ and $[o,q\rangle$.
Furthermore, the point $v=\left( 1+x_0, y_0 \right)$
is a point of the Busemann bisector of $[o,p\rangle$ and $[o,q\rangle$.
Thus, we have
\[
\frac{y_0}{1+x_0} = y_0+s_0(1-x_0),
\]
which leads to the differential equation
\[
\frac{-x}{1-x^2}y = y'.
\]
An elementary computation shows that the solutions of this differential equation are
$y=c \sqrt{1-x^2}$ for some $c > 0$.
Since $y=1$ is a tangent line of  $M$, we have that $c=1$, which, by the symmetry of $M$,
yields the assertion.

{\bf Acknowledgements.}
The author is indebted to E. Makai Jr., Sz. Szab\'o, B. Szil\'agyi and S. Tabachnikov for their valuable remarks regarding the topics in the paper, and particularly to an anonimous referee
whose comment led to a more general form of Theorem~\ref{thm:equaltangent} with a simpler proof.

\end{document}